\documentclass[11pt,reqno]{amsart}

\usepackage{mathtools}
\usepackage{amssymb}
\usepackage{hyperref}
\usepackage{imakeidx}
\usepackage{tikz}
\usepackage{manfnt}
\usetikzlibrary{%
  matrix,%
  calc,%
  arrows%
}
\usepackage{graphicx}
\usepackage{fullpage}

\newcommand{\C}{\mathbb{C}}
\newcommand{\F}{\mathbb{F}}

\newcommand{\N}{\mathbb{N}}
\renewcommand{\P}{\mathbb{P}}

\newcommand{\RR}{\mathbb{R}}

\newcommand{\Z}{\mathbb{Z}}
\newcommand{\comment}[1]{}

\theoremstyle{definition}
\newtheorem{theorem}{Theorem}[section]

\theoremstyle{definition}

\numberwithin{equation}{subsection}
\theoremstyle{plain}
\newtheorem{thm}{Theorem}[section]
\newtheorem{prop}[theorem]{Proposition}
\newtheorem{cor}[theorem]{Corollary}

\newtheorem{lem}[theorem]{Lemma}

\theoremstyle{definition}

\newtheorem{rem}[theorem]{Remark}
\newtheorem{exm}[theorem]{Example}

\newtheorem*{theorem*}{Theorem}
\newtheorem*{lem*}{Lemma}
\newtheorem*{con*}{Conjecture}
\newtheorem{Claim*}{Claim}
\newtheorem*{defn*}{Definition}

\newtheorem*{rem*}{Remark}

\makeatletter
\def\imod#1{\allowbreak\mkern10mu\left({\operator@font mod}\,\,#1\right)}
\makeatother

\begin{document}

\title[Uniform Symbolic Topologies via 
Multinomial Expansions]{Uniform Symbolic Topologies via 
Multinomial Expansions}
\author{Robert M. Walker}

\address{Department of Mathematics, University of Michigan, Ann Arbor, MI, 48109}
\email{robmarsw@umich.edu}

\parskip=10pt plus 2pt minus 2pt

\begin{abstract} When does a Noetherian commutative ring $R$ have uniform symbolic topologies on primes--read, when does there exist an integer $D>0$ such that the symbolic power $P^{(Dr)} \subseteq P^r$ for all prime ideals $P \subseteq R$ and all $r >0$? Groundbreaking work of Ein-Lazarsfeld-Smith, as extended by Hochster and Huneke, and by Ma and Schwede in turn, provides a beautiful answer in the setting of finite-dimensional excellent regular rings. It is natural to then sleuth for analogues where the ring $R$ is non-regular, or where the above ideal containments can be improved using a linear function whose growth rate is slower. This manuscript falls under the overlap of these research directions.    
Working with a prescribed type of prime ideal $Q$ inside of  tensor products of domains of finite type over an algebraically closed field $\F$, we present binomial- and multinomial expansion criteria for containments of type $Q^{(E r)} \subseteq Q^r$, or even better, of type $Q^{(E (r-1)+1)} \subseteq Q^r$ for all $r>0$. The final section consolidates remarks on how often we can  utilize these criteria, presenting an example. 
 \end{abstract}

\thanks{2010 \textit{Mathematics Subject Classification:} 13H10, 14C20, 14M25.
}
\thanks{\textit{Keywords:} multinomial theorem, non-isolated singularities, symbolic powers, toric variety.}
\maketitle


\section{Introduction and Conventions for the Paper}

Given a Noetherian commutative ring $R$, when is there an integer $D$, depending only on $R$, such that the symbolic power $P^{(D r)} \subseteq P^r$ for all prime ideals $P \subseteq R$ 
 and all positive integers $r$? 
  In short, when does $R$ have \textbf{uniform symbolic topologies} on primes \cite[Section 3]{5authorSymbolicSurvey} \cite{HKV2}? 

The Ein-Lazarsfeld-Smith Theorem  \cite{ELS}, as extended by Hochster and Huneke \cite{HH1}, says that if $R$ is a $d$-dimensional regular ring containing a field,  then $Q^{(Dr)} \subseteq Q^r$ for all radical ideals $Q \subseteq R$ and all $r > 0$, where $D = \max \{1 , d-1 \}$.\footnote{This result has been extended to all excellent regular rings, even in mixed characteristic, by Ma-Schwede \cite{MaSchwede17}.} To what extent does this theme ring true for non-regular rings? Under mild stipulations, a local domain $R$ regular on the punctured spectrum has uniform symbolic topologies on primes \cite[Cor.~3.10]{HKV}. 
This paper makes first strides in establishing affirmative cases of the above questions for rings with non-isolated singularities. We postpone sojourning into that wilderness until  Section \ref{section: Finale 1}, once we have proven our main result--Theorem \ref{thm:main000}.

We now revisit the regular setting. Over an arbitrary field $\F$, $S = \F [\mathbb{P}^N] = \F[x_0, x_1, \ldots, x_N]$ is a standard $\N$-graded polynomial ring. The groundbreaking work of Ein-Lazarsfeld-Smith and Hochster-Huneke \cite{ELS, HH1}  implies that the symbolic power $I^{(Nr)} \subseteq I^r$ for all graded ideals 
$0 \subsetneqq I \subsetneqq S$ and all integers $r > 0$. In particular, $I^{(4)} \subseteq I^2$ holds for all graded ideals in $\F[\P^2]$, and Huneke asked whether an improvement $I^{(3)} \subseteq I^2$ holds for any radical ideal $I$ defining a finite set of points in $\P^2$. 
Building on this, Harbourne proposed dropping the symbolic power from $Nr$ down to the \textbf{Harbourne-Huneke bound} $Nr - (N-1) = N (r-1)+1$ when $N \ge 2$ \cite[Conj. 8.4.2]{Primer}: i.e.,  
\begin{equation}\label{Harbourne-Huneke bound 001}
I^{(N (r-1) +1)} \subseteq I^r \mbox{ for any graded ideal $0 \subsetneqq I \subsetneqq S$, all }r > 0, \mbox{ and all }N \ge 2.
\end{equation} 
There are several scenarios where these improved containments hold: for instance, they hold for all monomial ideals in $S$ over any field \cite[Ex.~8.4.5]{Primer}; see also recent work of Grifo-Huneke \cite{GrifoHun00}. 

However, Dumnicki, Szemberg, and Tutaj-Gasi\'{n}ska showed in characteristic zero \cite{DSTG01} that the containment $I^{(3)} \subseteq I^2$ can fail for a radical ideal defining a point configuration in $\P^2$.  
Harbourne-Seceleanu showed in odd positive characteristic \cite{resurge2} that \eqref{Harbourne-Huneke bound 001} can fail for pairs $(N, r) \neq (2, 2)$ and ideals $I$ defining a point configuration in $\P^N$. 
 Akesseh \cite{Akes01} cooks up many new counterexamples to  \eqref{Harbourne-Huneke bound 001}  from these original constructions. 
No prime ideal counterexample has been found. 

Our goal is to establish Harbourne-Huneke bounds on the growth of symbolic powers of certain primes in non-regular Noetherian rings containing a field. This project first began with Theorems 1.1 and 3.1 in \cite{Walker002}: to clarify, normal affine semigroup rings 
 are domains generated by Laurent monomials that arise as the coordinate rings of 
  normal affine toric varieties \cite{torictome,introtoric}.
\begin{thm}[{\cite[Thm.~1.1]{Walker002}}]\label{thm: finite tensor products 000}
\textit{Let $R_1, \ldots, R_n$ be normal 
affine semigroup rings over a field $\F$, built, respectively, from full-dimensional strongly convex rational polyhedral cones $\sigma_i \subseteq \RR^{m_i}$, $1 \le i \le n$. 
For each $1 \le i \le n$,  suppose there is an integer $D_i>0$ such that $P^{(D_i (r-1) + 1)} \subseteq P^r$ 
for all $r> 0$ and all monomial primes $P \subseteq R_i$. Set $D := \max\{D_1, \ldots, D_n \}$. Then $Q^{(D (r-1) + 1)} \subseteq Q^r$ for all $r>0$ and any monomial prime $Q$ in the normal 
affine semigroup ring 
 $R = R_1 \otimes_\F \cdots \otimes_\F R_n$. 
}  
\end{thm}
\noindent To prove Theorem \ref{thm: finite tensor products 000}, we needed to know first that monomial primes in $R_i$ expand to monomial primes in $R$, that any $Q$ as above can be expressed as a sum $Q = \sum_{i=1}^n P_i R$ where each $P_i \subseteq R_i$ is a monomial prime, and that the symbolic powers of $Q$ admit a multinomial expansion in terms of symbolic powers of the $P_i R$. These ideas will resurge below, but in a more general setup.  

One drawback of Theorem \ref{thm: finite tensor products 000} is that it only covers a finite collection of prime ideals. What follows is the main result of this paper, a more powerful variant of Theorem \ref{thm: finite tensor products 000}  that will typically cover infinitely-many primes inside of tensor product domains; see Remark \ref{rem:infinite-spectrum} for details. 

\begin{thm}\label{thm:main000}
\textit{Let $\F$ be an algebraically closed field. Let $R_1, \ldots, R_n$ $(n \ge 2)$ be affine commutative $\F$-algebras which are domains. Suppose that for each $1 \le i \le n$, there exists a  positive integer $D_i$ such that for all prime ideals $P$ in $R_i$,  either: $(1)$ $P^{(D_i r) } \subseteq P^r$  for all $r>0$ and for all $i$; or, even stronger, $(2)$  $P^{(D_i (r-1) + 1)} \subseteq P^r$ for all $r>0$ and for all $i$. Fix any $n$ prime ideals $P_i$ in $R_i$, and consider the expanded ideals $P_i' = P_iT$ in the affine domain $T = (\bigotimes_{\F})_{i=1}^n R_i$, along with their sum $Q = \sum_{i=1}^n P_i'$ in $T$. 
When $(1)$ holds, $Q^{(D r) } \subseteq Q^r$ for all $r>0$, where $D  = D_1 + 
\cdots + D_n$.  When $(2)$ holds,  this improves to $Q^{(D (r-1)+1) } \subseteq Q^r$ for all $r>0$, where 
$D  = \max \{D_1 , \ldots, D_n \}$.} 
\end{thm}
\noindent The proof of this theorem leverages a multinomial formula for the symbolic powers of the prime ideal $Q$ in $T$ (Theorem \ref{thm:multinomial001}). 
H\`{a}, Nguyen, Trung, and Trung recently announced a binomial theorem for symbolic powers of ideal sums \cite[Thm.~3.4]{HNTTBinomial}, generalizing \cite[Thm.~7.8]{BCGHJNSVTV-000}, where one takes two arbitrary ideals $I \subseteq A, J \subseteq B$ inside of two Noetherian commutative algebras over a common field $k$, whose tensor product $R = A \otimes_k B$ is Noetherian; see Remark \ref{rem:Bocci-Ha-etal} for details. 
However, we give a proof of the Multinomial Theorem \ref{thm:multinomial001} which is more elementary and self-contained.

\noindent \textbf{Conventions:} All our rings are Noetherian and commutative with identity. 
Indeed, our rings will typically be \textit{affine} $\F$-algebras, that is, of finite type over a fixed field $\F$ of arbitrary characteristic. 
By \textit{algebraic variety}, we will mean an integral scheme of finite type over the field $\F$.  

\noindent \textbf{Acknowledgements:} I thank my thesis adviser Karen E. Smith, while on her sabbatical, and my surrogate adviser, Mel Hochster, for several patient and fruitful discussions during the Fall 2016 semester. I thank Huy T\`{a}i H\`{a} for sharing a preliminary draft of Section 2 of \cite{HNTTBinomial} in November 2016. 
I thank Elo\'{i}sa Grifo Pires,  Daniel Hern\'{a}ndez,  Jack Jeffries, Luis N\'{u}\~{n}ez-Betancourt, and Felipe P\'{e}rez  for reading a draft of the paper.  
I thank an anonymous referee for comments improving exposition in the paper.  I acknowledge support from a NSF GRF (Grant No.  PGF-031543), NSF RTG grant DMS-0943832, and a 2017 Ford Foundation Dissertation Fellowship.  

\section{A Multinomial Theorem for Symbolic Powers of Primes}\label{section: SymbMultinomial}


If $P$ is any prime ideal in a Noetherian ring $R$, its \textbf{$a$-th ($a \in \Z_{>0}$) symbolic power} ideal
$$P^{(a)} = P^a R_P \cap R = \left\lbrace f \in R \colon uf \in P^a \mbox{ for some }u \in R - P  \right\rbrace$$ is the $P$-primary component in any Lasker-Noether minimal primary decomposition of $P^a$; it is the smallest $P$-primary ideal containing $P^a$. We separately set $P^{(0)} = P^0 = R$ to be the unit ideal.  
Note that $P^{(1)} = P$, while the inclusion $P^{(a)} \supseteq P^a$ for each  $a>1$ can be strict. 
Before proceeding, we record a handy asymptotic conversion lemma. 

\begin{lem}[Cf., {\cite[Lem.~3.3]{Walker002}}]\label{lem: equivalence of symb power containments 001}
Given a prime ideal $P$ in a Noetherian ring $S$, and $E \in \Z_{> 0 }$, 
$$\mbox{$P^{(N)} \subseteq P^{\lceil N / E \rceil}$ for all $N \ge 0$ } \iff  \mbox{$P^{(E (r-1) + 1)} \subseteq P^r$ for all $r>0$}.$$ 
\end{lem}


\noindent \textbf{Torsion free modules over Noetherian Domains.} 
A module $M$ over a domain $R$ is \textbf{torsion free} if whenever $rx = 0$ for some $x\in M$ and $r\in R$, then either $r = 0$ or $x=0$. We first record a lemma on torsion free modules to be used both here and in the next subsection (cf., Lemmas 15.6.7-8 from the Stacks Project page \cite{Sta16} on 
 \href{http://stacks.math.columbia.edu/tag/0549}{torsion free modules}):
\begin{lem}\label{lem: characterization-torsionfreeness 00}
Let $R$ be a Noetherian domain. Let $M$ be a nonzero finitely generated R-module. Then the following assertions are equivalent:
\begin{enumerate}
\item M is torsion free;
\item M is a submodule of a finitely generated free module;
\item $(0)$ is the only associated prime of M, i.e., $\operatorname{Ass}_R (M) = \{(0)\}$.
\end{enumerate}
\end{lem} 
Working over an arbitrary field $\F$, we fix two affine $\F$-algebras $R$ and $S$ which are domains. The tensor product $T = R \otimes_\F S$ will be an affine $\F$-algebra. $T$ is a domain when $\F$ is algebraically closed (Milne  \cite[Prop.~4.15]{JSMilneAG}). 
We note that when $R$ and $S$ are duly nice (e.g., polynomial, or normal toric rings more generally), $T$ is a domain over any field. We now record two additional lemmas. 

\begin{lem}\label{lem:tensorproduct-torsionfreeness01}
Suppose that all three of $R$, $S$, and $T = R\otimes_{\mathbb F}S$ are affine domains over a field $\mathbb F$. If $M$ and $N$ are finitely generated torsion free modules over $R$ and $S$, respectively, then $M\otimes_{\mathbb F} N$ is a finitely generated torsion free $T$-module. 
\end{lem}

\begin{proof}
Viewed as vector spaces, $M \otimes_\F N = 0$ if and only if $M = 0$ or $N = 0$, in which case torsion freeness is vacuous. So we will assume all three of $M, N,$ and $M \otimes_\F N$ are nonzero. Per Lemma \ref{lem: characterization-torsionfreeness 00}, suppose we have embeddings $M\subseteq R^a$ and $N \subseteq S^b$. Apply the functor $\bullet \otimes_{\mathbb F} N$ to the first inclusion to get $M \otimes N \subseteq R^a \otimes N$, which in turn is contained in $ R^a \otimes S^b$
by tensoring the inclusion $N \subseteq S^b$ with  $R^a$. Thus  $M \otimes N \subseteq R^a \otimes S^b \cong (R\otimes S)^{ab} = T^{ab},$ where the isomorphism is easily checked in the category of $\F$-vector spaces since direct sum commutes with tensor product. Of course, this inclusion holds in the category of $T$-modules, and all $T$-submodules of $T^{ab}$ are finitely generated since $T$ is Noetherian, so we are done by invoking Lemma \ref{lem: characterization-torsionfreeness 00} again.
\end{proof}


\begin{lem}\label{lem:symbpower-torsionfree-quotient} For any prime $P$ in any Noetherian ring $A$, the finitely generated module $P^{(a)}/P^{(a+1)}$ is torsion free as an $A/P$-module for all integers $a \ge 0$.
\end{lem}
\begin{proof}
Say $\overline{x} \in (P^{(a)}/P^{(a+1)})$ is killed by $\overline{r} \in A/P$.
This means, lifting to $A$, that $x\in P^{(a)}$ and  $rx \in   P^{(a+1)}$. Localize at $P$.  Then $rx \in   P^{(a+1)}A_P = P^{a+1}A_P.$ If $r \not\in P$, this means $x \in   P^{a+1}A_P\cap A = P^{(a+1)}.$ 
That is, either $\overline r =0$ in $A/P$ or otherwise, $\overline x = 0$ in $(P^{(a)}/P^{(a+1)})$. Ergo by definition,  $(P^{(a)}/P^{(a+1)})$ is a torsion-free $A/P$-module.
\end{proof}


Finally, we record a consequence of Lemma  \ref{lem:tensorproduct-torsionfreeness01} that will be important in the next subsection. The following proposition follows immediately from Lemmas \ref{lem:tensorproduct-torsionfreeness01} and \ref{lem:symbpower-torsionfree-quotient} 

\begin{prop}\label{cor:torsionfree-tensorproduct 001}
Suppose that all three of $R$, $S$, and $T = R\otimes_{\mathbb F}S$ are affine domains over a field $\mathbb F$. 
Fix two prime ideals $P$ and $Q$ in $R$ and $S$ respectively, such that the affine $\F$-algebra $T' = (R/P) \otimes_\F (S/Q)$ is a domain. 
Then $(P^{(a)} / P^{(a+1)}) \otimes_\F (Q^{(b)}/ Q^{(b+1)})$ is finitely generated and torsion free  over $T'$ for any pair of nonnegative integers $a$ and $b$. 
\end{prop}

\noindent \textbf{Proving the Multinomial Theorem.} 
Working over an algebraically closed field $\F$, we fix two affine $\F$-algebras $R$ and $S$ that are domains, and two  prime ideals $P \subseteq R$, $Q \subseteq S$. Let 
$$T = R \otimes S \supseteq P \otimes S + R \otimes Q =: PT+QT, \quad T'  = (R/P) \otimes (S/Q) \cong T/(PT+QT),$$ where all tensor products are over $\F$. 
Both $T$ and $T'$ are affine domains over $\F$. Because $\F$ is algebraically closed, the extended ideals $PT, QT$ are both prime, along with their sum $PT + QT$.  
 We cannot relax the assumption that $\F$ is algebraically closed to its merely being perfect. 
For instance, $\RR$ is perfect (being of characteristic  zero), and along the ring extension 
$$S : = \frac{\RR[x]}{(x^2+1)} \cong \C  \hookrightarrow T : = \frac{\C[x]}{(x^2+1)} \cong \C \otimes_\RR S \cong \C \otimes_\RR \C$$ the zero ideal of $S$ (which is maximal) extends to a radical ideal which is not prime. 

Relative to a flat map $\phi \colon A \to B$ of Noetherian rings, we define the ideal $JB := \langle \phi(J) \rangle B$ for any ideal $J$ in $A$. Then $J^r B = (JB)^r$ for all $r \ge 0$, since the two ideals share a generating set. 
We define the set 
$\mathcal{P}(A) = \{\mbox{prime ideals }P \subseteq A  \colon 
PB \mbox{ is prime} \}$ to consist of prime ideals that extend along $\phi$ to prime ideals of $B$. We now record without proof a handy proposition.  

\begin{prop}[Cf., {\cite[Prop.~2.1]{Walker002}}]\label{prop: faithful flatness criterion 001}
Suppose $\phi \colon A \to B$ is a faithfully flat map of Noetherian rings. Then for each prime ideal $P \in \mathcal{P}(A)$ and all integer pairs $(N, r) \in (\Z_{\ge 0})^2$, we have 
\begin{equation}\label{eqn: faithful flatness criterion 001}
P^{(N)} B = (PB)^{(N)},
\end{equation} 
and 
$P^{(N)} \subseteq P^r$ if and only if $(PB)^{(N)} =  P^{(N)}B \subseteq P^r B = (PB)^r. $ 
\end{prop}
\noindent When $B$ is a polynomial ring in finitely many variables over $A$ and $\phi$ is inclusion, $\mathcal{P}(A) = \mbox{Spec}(A)$. It is possible that $\mathcal{P}(A ) \neq \operatorname{Spec}(A)$ in Proposition \ref{prop: faithful flatness criterion 001}, per the $T \cong \C \otimes_{\RR} \C$ example. Working over a field $\F$, we  use Proposition \ref{prop: faithful flatness criterion 001} when $B = A \otimes_\F C$ for two affine $\F$-algebras, 
so $B$ is an affine $\F$-algebra; when $A$ and $C$ are domains and $\F$ is algebraically closed, $B$ is a domain, $\mathcal{P}(A) = \operatorname{Spec}(A)$ and $\mathcal{P}(C) = \operatorname{Spec}(C)$. 

We are now ready to prove a binomial theorem for the  symbolic powers of $PT + QT$. 

\begin{theorem}\label{thm:binomialexpansion001}
\textit{For all $n \ge 1$, the symbolic power $(PT+QT)^{(n)} = \sum_{a+b=n} (PT)^{(a)} (QT)^{(b)}$.}
\end{theorem}

\begin{proof}
We'll drop the $T$'s from the notation, and we will assume that both $P, Q$ are nonzero to justify the effort. For $0 \le c \le n$, set 
$J_c = \sum_{t=0}^c P^{(c-t)} Q^{(t)}$, so $J_{c} \subseteq J_{c-1}$ for all $1 \le c \le n$, since $P^{(c-t)} \subseteq P^{(c-1 - t)}$ for $t \le c-1$ and for $t = c$, $Q^{(c)} \subseteq Q^{(c-1)}$. 
Note that $$(P+Q)^n = \sum_{a+b=n} P^{a} Q^{b} \subseteq J_n = \sum_{a+b= n} P^{(a)} Q^{(b)} \stackrel{(!)}{\subseteq} (P+Q)^{(n)},$$ 
and \textbf{(!)} is easy to verify term-by-term for each $P^{(a)}Q^{(b)}$.  Indeed, $P^{(a)}Q^{(b)}$ is generated by elements of the form $fg$ with 
 $f\in P^{(a)} \subset R$ and $g\in Q^{(b)} \subset S$ (viewing them as elements of $T$). We need $fg \in (P+Q)^{(a+b)}.$ 
Per Proposition \ref{prop: faithful flatness criterion 001}, there exist $u \in R - P$ and $v \in S -  Q$ such that $uf \in P^{a}$ and 
 $vg \in Q^{b}$. Viewing $u$ and $v$ as elements of the overring $T$, we have $uv \not\in (P+Q)$. Indeed, since $P+Q$ is prime, if $uv \in P+Q$, then either $u$ or $v$ is in $P+Q$, but $(P+Q)T\cap R = P$ and $(P+Q)T\cap S = Q,$ contradicting that $u\not\in P$ and $v\not\in Q$.
Therefore, in $T$,  $(uf)(vg) = (uv)(fg)  \in P^aQ^b \subset (P+Q)^{a+b}$, which means $fg \in  (P+Q)^{(a+b)}.$ Thus \textbf{(!)} holds,  and notably $J_n$ is a proper ideal--read, $J_n \subsetneqq T$. 

Since $J_n$ contains $(P+Q)^n $, and $(P+Q)^{(n)}$ is the smallest $(P+Q)$-primary ideal containing $(P+Q)^n$, the opposite inclusion to (!) will follow once we show that $J_n$ is $(P+Q)$-primary, i.e., that the set of associated primes $\mbox{Ass}_T (T/J_n) = \{P+Q\}$.  
We have short exact sequences  of $T$-modules
$$0 \to J_{c-1}/J_c \to T/J_c \to T/J_{c-1} \to 0, \quad \mbox{ for all }1 \le c \le n.$$
Thus $\mbox{Ass}_T (J_{c-1}/J_c) \subseteq \mbox{Ass}_T (T/J_c) \subseteq \mbox{Ass}_T (J_{c-1}/J_c) \cup \mbox{Ass}_T (T/J_{c-1})$ for all $1 \le c \le n$, using the fact that 
given an inclusion of modules $N \subseteq M$, 
$$\mbox{Ass}(N) \subseteq \mbox{Ass}(M) \subseteq \mbox{Ass}(N) \cup \mbox{Ass}(M/N).$$ 
Thus by iterative unwinding and using that $J_0 = T$, i.e., $\mbox{Ass}_T (T/J_0) = \varnothing$, we conclude that
\begin{align}\label{eqn:associatedsandwich00}
\varnothing \neq \mbox{Ass}_T (T/J_n) \subseteq \bigcup_{c=1}^n \mbox{Ass}_T (J_{c-1}/J_{c}).
\end{align}
Taking all direct sums and tensor products over $\F$, we have a series of vector space isomorphisms 
\begin{align}\label{eqn:symbolicquotientisom00}
J_{c-1} / J_{c} \cong \bigoplus_{a+b = c-1} [P^{(a)}/ P^{(a+1)} \otimes Q^{(b)} / Q^{(b+1)}], \quad 1 \le c \le n.
\end{align}
We prove this first, considering two chains of symbolic powers, where each ideal is expressed as a direct sum of $\F$-vector spaces:
\begin{align*}
P^{(c)} = V_0 \subseteq P^{(c-1)} = V_0 \oplus V_1 &\subseteq \ldots \subseteq P^{(0)} = R = V_0 \oplus \cdots \oplus V_{c}, \\
Q^{(c)} = W_0 \subseteq Q^{(c-1)} = W_0 \oplus W_1 &\subseteq \ldots \subseteq Q^{(0)} = S = W_0 \oplus \cdots \oplus W_{c}.
\end{align*}
In particular, for all pairs $0 \le a, b \le c-1$, 
\begin{align*}
P^{(a)} = \bigoplus_{i=0}^{c-a} V_i, \quad  P^{(a+1)} = \bigoplus_{i=0}^{c-a-1} V_i, \quad  
Q^{(b)} = \bigoplus_{j=0}^{c-b} W_j, \quad  Q^{(b+1)} = \bigoplus_{j=0}^{c-b-1} W_j.
\end{align*}
For any pair $a, b$ as above with $a+b = c-1$,  $c - b = a+1$, and so 
$$\bigoplus_{a+b = c-1} \frac{P^{(a)}}{P^{(a+1)}} \otimes_\F  \frac{Q^{(b)}}{Q^{(b+1)}} \cong \bigoplus_{a+b=c-1} V_{c-a} \otimes W_{c-b} =  \bigoplus_{a=0}^{c-1} V_{c-a} \otimes W_{a+1}.$$
We now prove \eqref{eqn:symbolicquotientisom00} by  killing off a common vector space. 
First,
\begin{align*}
J_{c-1} = \sum_{a+b=c-1} P^{(a)} Q^{(b)} &= \bigoplus_{\stackrel{0 \le a \le c-1}{0 \le i \le c-a,  0 \le j \le a+1}} V_i \otimes W_j \\ 
&= \boxed{\bigoplus_{\stackrel{0 \le a \le c-1}{0 \le i < c-a \mbox{ or } 0 \le j < a+1}} (V_i \otimes W_j) } 
\oplus 
\bigoplus_{a=0}^{c-1} V_{c-a} \otimes W_{a+1} , \\
\mbox{ while } \quad J_{c} = \sum_{a+b=c} P^{(a)} Q^{(b)} &= \boxed{\bigoplus_{\stackrel{0 \le a \le c}{0 \le i \le c-a, 0 \le j \le a}} V_i \otimes W_j. }  
\end{align*}
Identifying repeated copies of a $V_i \otimes V_j$ term with $i+ j \le c$ (we can do this since we are working with vector subspaces of the ring $T$), it is straightforward to check that the boxed sums are equal. Thus for each $1 \le c \le n$, we have canonical isomorphisms of $\F$-vector spaces:
\begin{align*}
J_{c-1}/J_c \cong \bigoplus_{a=0}^{c-1} V_{c-a} \otimes W_{a+1} \cong \bigoplus_{a+b = c-1} \frac{P^{(a)}}{P^{(a+1)}} \otimes_\F  \frac{Q^{(b)}}{Q^{(b+1)}} .
\end{align*}

Therefore, since for each $1 \le c \le n$ there is a natural surjective $T$-module map (hence $\F$-linear)
$$
\bigoplus_{a+b = c-1} [P^{(a)}/ P^{(a+1)} \otimes Q^{(b)} / Q^{(b+1)}] \rightarrow J_{c-1} / J_{c}, 
$$
this map must be injective per isomorphism \eqref{eqn:symbolicquotientisom00}.
Thus for all $1 \le c \le n$,  
$$\mbox{Ass}_T (J_{c-1}/J_{c}) = \bigcup_{a+b = c-1} \mbox{Ass}_T  [P^{(a)}/ P^{(a+1)} \otimes Q^{(b)} / Q^{(b+1)}].$$
For any $1 \le c \le n$ such that $J_{c-1}/J_c \neq 0$, i.e., $\operatorname{Ass}_T (J_{c-1}/J_c) \neq \varnothing$, in turn the above identity implies that one of the modules $P^{(a)}/ P^{(a+1)} \otimes Q^{(b)} / Q^{(b+1)}$ is nonzero, in which case  
\begin{equation}\label{eqn: associatedprimesidentity00}
\mbox{Ass}_T (J_{c-1}/J_{c}) = \bigcup_{a+b = c-1} \mbox{Ass}_T  [P^{(a)}/ P^{(a+1)} \otimes Q^{(b)} / Q^{(b+1)}] = \{P+Q\}.
\end{equation} To explain the  right-hand equality: for any pair $(a, b) \in (\Z_{\ge 0})^2$, Proposition  \ref{cor:torsionfree-tensorproduct 001} says that $$M_{a, b} : = P^{(a)}/ P^{(a+1)} \otimes Q^{(b)} / Q^{(b+1)}$$ is a finitely generated torsion-free module over $T' = (R/P) \otimes (S/Q) \cong T / (P+Q)$; 
thus when $M_{a, b} \neq 0$, we have $\mbox{Ass}_{T/(P+Q)} (M_{a, b}) = \{(0)\}$ by Lemma \ref{lem: characterization-torsionfreeness 00}: that is, $\mbox{Ass}_{T} (M_{a, b}) = \{P+Q\}$.  

Finally, combining \eqref{eqn: associatedprimesidentity00} with the inclusion  \eqref{eqn:associatedsandwich00} for $\mbox{Ass}_T (T/J_n) \neq \varnothing$--recall, $J_n$ is a proper ideal,   
we conclude that $\mbox{Ass}_T (T/J_n) = \bigcup_{c=1}^n \mbox{Ass}_T (J_{c-1}/J_{c}) = \{P+Q\},$ 
that is, the ideal $J_n$ is $(P+Q)$-primary as was to be shown. Thus $J_n \supseteq (P+Q)^{(n)},$ and indeed this is an equality. 
\end{proof}

We now deduce a multinomial theorem by induction on the number of tensor factors:

\begin{theorem}\label{thm:multinomial001}
\textit{Let $\F$ be an algebraically closed field. Let $R_1, \ldots, R_n$ $(n \ge 2)$ be affine commutative $\F$-algebras which are domains. Fix any $n$ prime ideals $P_i$ in $R_i$, and consider the expanded ideals $P_i' = P_iT$ in the affine domain $T = (\bigotimes_{\F})_{i=1}^n R_i$. 
Then the symbolic power  
\begin{equation}\label{eqn:multinomialexpansion01}
\left(\sum_{i=1}^{n} P_i'\right)^{(N)} =  \sum_{A_1 + \cdots + A_n =N} \prod_{i=1}^n (P_i')^{(A_i)} \mbox{ for any $N \ge 0$}.
\end{equation}
}
\end{theorem}

\begin{proof} 
Induce on the number $n$ of tensor factors with base case $n=2$ being Theorem \ref{thm:binomialexpansion001}. Now suppose $n \ge 3$, and assume the result for tensoring up to $n-1$ factors. Suppose that $R = R_1 $ and $S =  R_2 \otimes_\F \cdots \otimes_\F R_n$, and that we have an expansion result in $S$ of the form
\begin{equation}\label{eqn:multinomial-induction001}
\left(\sum_{i=2}^n P_i\right)^{(N)} = \sum_{A_2 + \ldots + A_n = N} \prod_{i=2}^{n} P_i^{(A_i)} \quad \mbox{ for all nonnegative integers }N 
\end{equation}
for $n-1$ primes $P_i \subseteq R_i$  ($2 \le i \le n$). 
The sum $Q : = \sum_{i=2}^n P_i$ is prime along with all extensions of the  $P_i$ to $S$. 
Given a prime $P = P_1$ in $R$, the sum $P+Q$ is prime in $T = R \otimes_\F S$, together with all extensions $P_i T$ and $Q T$ being prime. The first equality below holds by Theorem \ref{thm:binomialexpansion001}, and  applying Proposition \ref{prop: faithful flatness criterion 001} to the extension $\phi \colon S \hookrightarrow T$, the second equality holds by \eqref{eqn:multinomial-induction001}:
\begin{align*}
(P+Q)^{(N)} = \sum_{A_1 + B = N} P^{(A_1)} Q^{(B)} &= \sum_{A_1=0}^N P_1^{(A_1)}  \left(\sum_{A_2 + \ldots + A_n = N-A_1} \prod_{i=2}^{n} P_i^{(A_i)} \right) \\
&\subseteq \sum_{A_1 +A_2 + \ldots + A_n = N}  \prod_{i=1}^{n} P_i^{(A_i)}, 
\end{align*}
using the fact that $I (J+K) \subseteq I J + IK$ whenever $I, J , K$ are ideals in a commutative  ring. This proves the $n$-fold version of the hard inclusion in the proof of Theorem \ref{thm:binomialexpansion001}; deducing the opposite inclusion is about as easy as before, hence the above inclusion is an equality. 
\end{proof}

\noindent \textbf{Proving Theorem \ref{thm:main000}:} We now use the Multinomial Theorem \ref{thm:multinomial001} to deduce a corollary. Note that Theorem \ref{thm:main000} is the version of this corollary where all tensor factors are assumed to satisfy uniform symbolic topologies on primes. 

\begin{cor}\label{cor:harbhun001}
Let $\F$ be an algebraically closed field. Let $R_1, \ldots, R_n$ $(n \ge 2)$ be affine commutative $\F$-algebras which are domains. Fix $n$ primes $P_i \subseteq R_i$, and consider the expanded ideals $P_i' = P_iT$ in the affine domain $T = (\bigotimes_{\F})_{i=1}^n R_i$; set $Q = \sum_{i=1}^{n} P_i'$.  
Suppose that for each $1 \le i \le n$, there exists a  positive integer $D_i$ such that either:  
\begin{enumerate}
\item $P_i^{(D_i r) } \subseteq P_i^r$  for all $r>0$ and for all $i$; or, even stronger,  
\item $P_i^{(D_i (r-1) + 1)} \subseteq P_i^r$ for all $r>0$ and for all $i$. 
\end{enumerate}
When $(1)$ holds, $Q^{(D r) } \subseteq Q^r$ for all $r>0$, where $D  = D_1 + 
\cdots + D_n$.  When $(2)$ holds, this improves to $Q^{(D (r-1)+1) } \subseteq Q^r$ for all $r>0$, where 
$D  = \max \{D_1 , \ldots, D_n \}$. 
\end{cor}

\begin{proof}
Assume (1) holds. Per Theorem \ref{thm:multinomial001} note that for $D = D_1+ D_2 +\cdots + D_n$,  
$$Q^{(Dr)} = \sum_{A_1 + A_2 + \cdots + A_n = D_1 r + D_2 r + \cdots + D_n r} \prod_{i=1}^n (P_i')^{(A_i)}.$$
In each $n$-tuple of indices $(A_1, \ldots, A_n)$, we must have that $A_j \ge D_j r$ for some $j$,  otherwise $\sum_{i=1}^n A_i < \sum_{i=1}^n D_i r$, a contradiction. Thus each summand $\prod_{i=1}^n (P_i')^{(A_i)}$ will lie in some $(P_j')^r$ applying (1) and Proposition \ref{prop: faithful flatness criterion 001}, and hence also in $Q^r$. Since $r > 0$ was arbitrary, we win.  

If (2) holds, then $P_i^{(D (r-1) + 1)} \subseteq P_i^r$ for all $r>0$ and all $i$, where $D = \max_{1 \le i \le n}  D_i$, so  equivalently per Lemma \ref{lem: equivalence of symb power containments 001} and Proposition \ref{prop: faithful flatness criterion 001}, for all $n$-tuples $(A_1, \ldots, A_n) \in (\Z_{\ge 0})^n$, we have containments 
$(P_i')^{(A_i)} \subseteq (P_i')^{\lceil A_i/D \rceil} \subseteq Q^{\lceil A_i/D \rceil}.$ 
 For all nonnegative integers $N$, per  Theorem \ref{thm:multinomial001}  
$$Q^{(N)} = \sum_{A_1+\cdots + A_n= N} \prod_{i=1}^n (P_i')^{(A_i)} \subseteq \sum_{A_1+\cdots + A_n = N} \prod_{i=1}^n (P_i')^{\lceil A_i/D \rceil} \subseteq \sum_{A_1+\cdots + A_n = N} \prod_{i=1}^n Q^{\lceil A_i/D \rceil}  \subseteq Q^{\lceil N/D \rceil},$$ since the integer $\sum_{i=1}^n \lceil A_i/D \rceil  \ge \lceil (\sum_{i=1}^n A_i)/D \rceil = \lceil N/D \rceil$ for all $n$-tuples $(A_1, \ldots, A_n) \in (\Z_{\ge 0})^n$  with $\sum_{i=1}^n A_i = N$. Thus equivalently, $Q^{(D (r-1)+1) } \subseteq Q^r$ for all $r>0$ by Lemma \ref{lem: equivalence of symb power containments 001}.
\end{proof}

\begin{rem}\label{rem:an-alternative-containment}
We get a much stronger conclusion in Corollary \ref{cor:harbhun001} when (2) holds. This is because we can then give a proof using Lemma \ref{lem: equivalence of symb power containments 001} as a workaround. It is less clear what the strongest conclusion to shoot for is when (1) holds.  We note that if (1) holds under Corollary \ref{cor:harbhun001}, then setting $D = \max D_i$, one can alternatively prove by contradiction that $$Q^{(n(Dr-1) + 1)} \subseteq Q^r  \mbox{ for all } r>0.$$ In part (1) of the proof above, simply adjust the claim ``$A_j \ge D_j r$ for some $j$" to ``$A_j \ge D r$ for some $j$." Otherwise, some tuple satisfies $n(Dr-1) +1 = \sum_{i=1}^n A_i \le n (Dr-1)$, a contradiction.   
\end{rem}

\begin{rem}\label{rem:infinite-spectrum}
When all hypotheses are satisfied, Corollary \ref{cor:harbhun001} typically applies to an infinite set of prime ideals in the tensor product $T$. If $R$ is a Noetherian ring of dimension at least two, or a Noetherian ring of dimension one which has infinitely many maximal ideals, then $\operatorname{Spec}(R)$ is infinite; see 
\cite[Exercises~21.11-21.12]{Alt-Klei13}. 
Now suppose $R_1, \ldots, R_n$ are $\F$-affine domains, with $n \ge 2$ and $\F$ algebraically closed, at least one of which is of dimension one or more.  Then in the domain $T = (\bigotimes_{\F})_{i=1}^n R_i$,  
the following set 
$\mathcal{Q}_{ED} (T)  := \{Q = \sum_{i=1}^n P_i T \in \operatorname{Spec}(T) \colon \forall 1 \le i \le n, P_i \in \operatorname{Spec}(R_i)\}$ is infinite.    
\end{rem}

\begin{rem}\label{rem:Bocci-Ha-etal}
Let $R = \F [x_1, \ldots, x_m]$, $S = \F [y_1, \ldots, y_n]$, and $T = R \otimes_\F S \cong \F [x_1, \ldots, x_m, y_1, \ldots, y_n]$ be polynomial rings over a field $\F$. Our original inspiration for Theorem \ref{thm:multinomial001} was the following 
\begin{theorem*}[Thm. 7.8 of Bocci et al \cite{BCGHJNSVTV-000}]
Let $I \subseteq R$ and $J \subseteq S$ be squarefree monomial ideals. Let $I' = IT$ and $J' = JT$ be their expansions to $T$. Then for any $N \ge 0$, the symbolic power  
\begin{equation}\label{eqn:binomialexpansion01}
(I' + J')^{(N)} = \sum_{i=0}^{N} (I')^{(N-i)} (J')^{(i)} = \sum_{A+B=N} (I')^{(A)} (J')^{(B)}.
\end{equation}
\end{theorem*}
\noindent Ha, Nguyen, Trung, and Trung \cite[Thm. 3.4]{HNTTBinomial} recently extended the above theorem to the case of two nonzero ideals $I \subseteq R, J \subseteq S$ in two Noetherian commutative $\F$-algebras such that $T = R \otimes_\F S$ is also Noetherian. A general multinomial theorem then follows  by adapting the proof of Theorem \ref{thm:multinomial001}, where one containment would require the  $n$-fold version of \cite[Lem. 2.1(i)]{HNTTBinomial}. 
Combining this multinomial expansion with the general versions of Lemma \ref{lem: equivalence of symb power containments 001} and Proposition \ref{prop: faithful flatness criterion 001} we proved in \cite[Prop. 2.1, Lem. 3.3, Prop. 3.4]{Walker002}, one can extend Corollary \ref{cor:harbhun001} to a form allowing, for instance, any proper ideals $I_i \subseteq R_i$. 
As a final note in passing, the proof of \cite[Prop. 2.1]{Walker002} still works up to a tweak of multiplicative system, for those who opt to define symbolic powers of proper ideals using only minimal associated primes as in \cite{HNTTBinomial}, rather than  using all associated primes as in \cite{Walker002}.    
\end{rem}

\section{Finale: Sample Applications to Tensor Power Domains}\label{section: Finale 1}

We begin with two results on uniform linear bounds on asymptotic growth of symbolic powers for equicharacteristic Noetherian domains that have nice structure, but need not be regular. The first is due to Huneke, Katz, and Validashti, while the second is due to  Ajinkya A. More. 

\begin{theorem}[{\cite[Cor.~3.10]{HKV}}]\label{thm:HKVCor3.10}
\textit{Let $R$ be an equicharacteristic Noetherian local domain such that $R$ is an isolated singularity. Assume that $R$ is either essentially of finite type over a field of characteristic zero or $R$ has positive characteristic, is $F$-finite and analytically irreducible. Then there exists an $E \ge 1$ such that $P^{(E r)} \subseteq P^r$ for all $r>0$ and all prime ideals $P$ in $R$.}  
\end{theorem}

\begin{theorem}[{\cite[Thm.~4.4, Cor.~4.5]{AAMore13}}, see also {\cite[Thm.~3.25]{5authorSymbolicSurvey}}]\label{thm:AAMore13} \textit{Suppose $R \subseteq S$ is a finite extension of equicharacteristic normal domains such that: $(1)$ $S$ is a regular ring generated as an $R$-module by $n$ elements, and $n!$ is invertible in $S$; and
$(2)$  $R$ is either essentially of finite type over an excellent Noetherian local ring (or over $\Z$), or is characteristic $p>0$ and $F$-finite.  
Then there exists an $E \ge 1$ such that $P^{(E r)} \subseteq P^r$ for all $r>0$ and all prime ideals $P$ in $R$.}  
\end{theorem}

\begin{rem}\label{rem:AG-isolated-sings} Suppose that $R$ is the coordinate ring of an affine variety over any perfect field $\F$, whose singular locus is zero dimensional. Then in tandem with the results of Ein-Lazarsfeld-Smith and Hochster-Huneke \cite{ELS,HH1},  Theorem \ref{thm:HKVCor3.10} would yield a uniform slope $E$ for all primes in $R$. In particular, this covers 
$\F[x,y,z]/(y^2-xz)$, 
$ \F[x,y,z,w]/(xy-zw)$, and more generally when $R$ corresponds to the affine cone over any smooth projective variety. 
\end{rem}

\noindent Indeed, the class of rings $R$ to which  Theorems \ref{thm:HKVCor3.10}-\ref{thm:AAMore13} apply is large. Applying Theorem \ref{thm:main000} to any collection of two or more rings under Remark \ref{rem:AG-isolated-sings}, Remark \ref{rem:infinite-spectrum} says that we can create an infinite set as a vantage point for data suggestive of uniform symbolic topologies in the corresponding tensor product domain. At present, since the domain we create has non-isolated singularities, there is no theorem in the literature affirming that the domain has uniform symbolic topologies on all primes. We illustrate how these matters occur together in an example below. 

But first, we fix an algebraically closed field $\F$. If $R$ is an $\F$-affine domain, 
we use the tensor power notation $T = R^{\otimes N} = (\bigotimes_\F)_{i=1}^N R_i$ to denote the $\F$-affine domain obtained by tensoring together $N$ copies of $R$ over $\F$, where $R_i$ and $R_j$ are presented as quotients of polynomial rings in disjoint sets of variables when $i<j$.  
From Remark \ref{rem:infinite-spectrum}, we recall the following set of prime ideals   $\mathcal{Q}_{ED}(T) := \{Q = \sum_{i=1}^N P_i T \in \mbox{Spec}(T) \colon \mbox{ each }P_i \in  \operatorname{Spec}(R_i)\}$. 


\begin{exm}\label{ex:hypersurfaceUSTP}
\textit{
To start, we fix an algebraically closed field $\F$. Given integers $a$ and $d$ both at least two, 
consider an affine hypersurface domain $R = \F[z_1,\ldots,z_a]/(F_d(\bar{z}))$ where $F_d$ is an irreducible homogeneous polynomial of degree $d$, with isolated singularity at the origin. 
Consider the varieties  $V_R = \operatorname{Spec}(R) \subseteq \F^a$ and $V = \operatorname{Spec}(T) \subseteq \F^{aN}$ where $$T = R^{\otimes N}= \frac{\F[z_{i, 1},\ldots, z_{i, a} \colon 1 \le i \le N]}{(F_d(z_{i, 1},\ldots, z_{i, a}) \colon 1 \le i \le N)}.$$
Per Remark \ref{rem:AG-isolated-sings}, Theorem \ref{thm:main000} implies that 
$Q^{(N E \cdot r)} \subseteq Q^r$ for all $r>0$ and all primes $Q \in \mathcal{Q}_{ED}(T)$. Meanwhile, in terms of $n$-factor Cartesian products, the singular locus $$ \operatorname{Sing}(V) = (\{0\}  \times V_R \times \cdots \times V_R) \cup (V_R \times \{0\} \times V_R \times \cdots \times V_R) \cup \cdots \cup (V_R \times \cdots \times V_R \times \{0\})$$
is equidimensional of dimension $(a-1)(N-1)$. 
  In particular, while $T$ is not an isolated singularity when $N \ge 2$, the set $\mathcal{Q}_{ED}(T)$ is infinite by Remark \ref{rem:infinite-spectrum} and provides a vantage point for witnessing uniform linear bounds lurking for the asymptotic growth of symbolic powers of primes in $T$.} 
\end{exm}


\begin{rem}
For pointers to results where an explicit value $E \ge 1$ as in Theorems \ref{thm:HKVCor3.10}-\ref{thm:AAMore13} is given for particular examples of domains $R$ under Remark \ref{rem:AG-isolated-sings}, we invite the reader to see the recent survey paper \cite[Thm.~3.29, Cor.~3.30]{5authorSymbolicSurvey}, along with the main results featured in the introductions to our papers \cite{Walker001,Walker002,Walker004}. Section 3 or 4 in each of the latter papers typically includes remarks about when the designated value $E$ can be considered optimal. 
\end{rem}

\noindent \textbf{Closing Remarks.}  
Launching from Theorem \ref{thm: finite tensor products 000} in the introduction, we have deduced a more powerful criterion for proliferating uniform linear bounds on the growth of symbolic powers of prime ideals (e.g., Harbourne-Huneke bounds)--Theorem \ref{thm:main000}.  
In the setting of domains of finite type over algebraically closed fields, this criterion contributes further evidence for Huneke's philosophy in \cite{hun1} about uniform bounds lurking throughout commutative algebra. 
We close with a goalpost question that exceeds our grasp at present: Given the role of tensor products in our manuscript, do analogues of the above criteria hold for other product constructions in commutative algebra, such as Segre products of $\N$-graded rings, or fiber products of toric rings?


\begin{thebibliography}{1}


\bibitem{Akes01} S. Akesseh.  {\em Ideal Containments Under Flat Extensions.}  \href{http://arxiv.org/abs/1512.08053}{arXiv/1512.08053} 

\bibitem{Alt-Klei13} A. Altman and S. Kleiman. {\em A Term of Commutative Algebra}. Worldwide Center of Mathematics LLC, Cambridge, MA, 2014.   
       
    

\bibitem{Primer} T. Bauer, S. Di Rocco, B. Harbourne, M. Kapustka, A.L. Knutsen, W. Syzdek, T. Szemberg. {\em A primer on Seshadri constants.} Contemporary Mathematics \textbf{496}  (2009), pp. 33-70. \href{https://arxiv.org/abs/0810.0728}{arXiv/0810.0728}



  


\bibitem{BCGHJNSVTV-000} C. Bocci, S. Cooper, E. Guardo, B. Harbourne, M. Janssen, U. Nagel, A. Seceleanu, A. Van Tuyl, T.  Vu. {\em The Waldschmidt constant for squarefree monomial ideals.} J. Algebr. Comb. \textbf{44}, no.4, pp. 875-904.    \href{http://arxiv.org/abs/1508.00477}{arXiv/1508.00477}



  \bibitem{torictome} D.A. Cox, J.B Little, and H.K. Schenck. {\em Toric Varieties} (2011), Graduate Studies in Mathematics 124.\\ American Mathematical Society, Providence, RI.


\bibitem{5authorSymbolicSurvey} H. Dao, A. De Stefani, E. Grifo, C. Huneke, L. Nu\~{n}ez-Betancourt. {\em Symbolic Powers of Ideals.} To appear in \textit{Advances in Singularities and Foliations: Geometry, Topology and Applications}, Springer Proceedings in Mathematics \& Statistics. \href{https://arxiv.org/abs/1708.03010}{arXiv/1708.0301} 

  
\bibitem{DSTG01} M. Dumnicki, T. Szemberg, and H. Tutaj-Gasi\'{n}ska. {\em Counterexamples to the $I^{(3)} \subseteq I^2$ containment.}  
J. Algebra \textbf{393} (2013) pp.24-29. \href{http://arxiv.org/abs/1301.7440}{arXiv/1301.7440}

\bibitem{ELS}  L. Ein, R. Lazarsfeld, and K. Smith. {\em Uniform bounds and  symbolic powers on smooth varieties}.  Invent. Math. \textbf{144} (2001), pp. 241-252. \href{https://arxiv.org/abs/math/0005098}{arXiv/0005098}

  








 \bibitem{introtoric}  W. Fulton. {\em Introduction to Toric Varieties} (1993), Annals of Math. Studies 131. Princeton University Press, Princeton, NJ. 
  


\bibitem{GrifoHun00} E. Grifo and  C. Huneke. {\em Symbolic powers of ideals defining F-pure and strongly F-regular rings.} International Mathematics Research Notices, rnx213, \href{https://doi.org/10.1093/imrn/rnx213}{https://doi.org/10.1093/imrn/rnx213} \href{https://arxiv.org/abs/1702.06876}{arXiv/1702.06876} 




\bibitem{HNTTBinomial} H.T. H\`{a}, H.D. Nguyen, N.V. Trung, T.N. Trung. {\em Symbolic Powers of Sums of Ideals.} \href{https://arxiv.org/abs/1702.01766}{arXiv/1702.01766}


   \bibitem{resurge2} B. Harbourne and A. Seceleanu. {\em Containment counterexamples for ideals of various configurations of points in $\mathbb{P}^n$.} J. Pure Appl. Algebra \textbf{219} (2015), no.4, pp. 1062-1072. \href{http://arxiv.org/abs/1306.3668}{arXiv/1306.3668}

   
   
   
  
\bibitem{HH1}  M. Hochster and C. Huneke. {\em Comparison of ordinary and symbolic powers of ideals}.  Invent. Math. \textbf{147} (2002), pp. 349-369. \href{https://arxiv.org/abs/math/0211174}{arXiv/0211174}
  

\bibitem{hun1}  C. Huneke. {\em Uniform bounds in Noetherian rings.}  Invent. Math. \textbf{107} (1992), pp. 203-223.
    
\bibitem{HKV}  C. Huneke, D. Katz, and J. Validashti. {\em Uniform equivalence of symbolic and adic topologies}. Illinois J. Math. \textbf{53} (2009), no. 1, pp. 325--338. 
    
\bibitem{HKV2}  C. Huneke, D. Katz, and J. Validashti. {\em Uniform symbolic topologies and finite extensions}. J. Pure Appl. Algebra \textbf{219} (2015), no. 3, pp. 543--550.
     
     
     

\bibitem{MaSchwede17} L. Ma and K. Schwede.  {\em Perfectoid multiplier/test ideals in regular rings and bounds on symbolic powers.}  \href{https://arxiv.org/abs/1705.02300}{arXiv/1705.02300} 


\bibitem{JSMilneAG} J.S. Milne. {\em Algebraic Geometry.} Version 5.22, 2013.  \href{http://www.jmilne.org/math/CourseNotes/ag.html}{http://www.jmilne.org/math/CourseNotes/ag.html}

\bibitem{AAMore13} A.A. More. {\em Uniform bounds on symbolic powers.} J. Algebra \textbf{383} (2013), pp. 29-41. 

     


\bibitem{Sta16} The Stacks Project authors. {\em The Stacks Project}, 2016. 



   




\bibitem{Walker001} R.M. Walker. {\em Rational Singularities and Uniform Symbolic Topologies.} Illinois J. Math. Vol. \textbf{60}, no. 2, Summer 2016, pp. 541-550.  \href{http://arxiv.org/abs/1510.02993}{arXiv/1510.02993} 

\bibitem{Walker002} R.M. Walker. {\em Uniform Harbourne-Huneke Bounds via Flat Extensions.}  \href{https://arxiv.org/abs/1608.02320}{arXiv/1608.02320}


\bibitem{Walker004} R.M. Walker. {\em Uniform Symbolic Topologies in Normal Toric Rings.}  \href{https://arxiv.org/abs/1706.06576}{arXiv/1706.06576}

  
   
  
\end{thebibliography}
\end{document}